\documentclass[leqno]{amsart}
\usepackage{amsmath}
\usepackage{amssymb}
\usepackage{amsthm}
\usepackage{enumerate}
\usepackage[mathscr]{eucal}
\theoremstyle{plain}
\newtheorem{theorem}{Theorem}[section]
\newtheorem{lemma}[theorem]{Lemma}

\theoremstyle{definition}

\newtheorem{remark}[theorem]{Remark}

\newtheorem{example}[theorem]{Example}

\newtheorem{cor}[theorem]{Corollary}
\theoremstyle{remark}

\begin{document}

\title [Numerical radius inequalities of operator matrices ]{Numerical radius inequalities of operator matrices with applications} 

\author{Pintu Bhunia, Santanu Bag and Kallol Paul}

\address{(Bhunia)Department of Mathematics, Jadavpur University, Kolkata 700032, India}
\email{pintubhunia5206@gmail.com}

\address{(Bag)Department of Mathematics, Vivekananda College For Women, Barisha, Kolkata 700008, India}
\email{santanumath84@gmail.com}

\address{(Paul)Department of Mathematics, Jadavpur University, Kolkata 700032, India}
\email{kalloldada@gmail.com}

\thanks{The first author would like to thank UGC, Govt. of India for the financial support in the form of junior research fellowship.}



\date{}
\maketitle
\begin{abstract}
 We present  upper and lower bounds for the numerical radius of $2 \times 2$ operator matrices which improve on the existing bounds for the same. As an application of the results obtained we give a better estimation for the zeros of a polynomial.

\end{abstract}

\section{Introduction}
\noindent Let $B(H)$ denote the $C^*$-algebra of all bounded linear operators on a complex Hilbert space $H$ with usual inner product $\langle .,. \rangle$.  Let $H_1, H_2$ be two Hilbert spaces and $B(H_1, H_2)$ be the set of all bounded linear operators from $H_1$ into $H_2$. If $H_1=H_2=H$ then we write $B(H_1, H_2)=B(H)$. For $T\in B(H)$, the operator norm of $T$, denoted as $\|T\|$, is defined as $\|T\|=\sup\{\|Tx\|:x \in H, ~\|x\|=1\}$. The numerical range of $T$, denoted as $W(T),$ is defined as $ W(T)=\{ \langle Tx,x \rangle: x \in H,~ \|x\|=1\}$. The spectrum of $T$, denoted as $\sigma(T)$ is defined as the collection of all spectral values of $T.$ The numerical radius and spectral radius of $T$, denoted as $w(T)$ and $\rho(T)$ respectively,  are defined as the radius of the smallest circle with centre at origin which contains the numerical range and spectrum of $T$.  It is well known that $\sigma(T) \subseteq \overline{W(T)}$ and so  $ \rho(T) \leq w(T). $ The Crawford number of $T$, denoted as $m(T),$ is defined as $m(T) = \inf \{ |\langle Tx,x \rangle |: x \in H, \|x\|=1\}.$  It is easy to see that $w(T)$ is a norm on $B(H)$,  equivalent to the operator norm that satisfies the inequality  
\[ \frac{1}{2}\|T\|\leq w(T)\leq \|T\|. \]
The first inequality becomes an equality if $T^2=0$ and the second inequality becomes an equality if $T$ is normal. Various numerical radius inequalities improving this inequality have been studied in \cite{OFK, FK2, PB1, KS, TY}. Kittaneh in \cite{FK3, FK2} respectively, proved that if $T\in B(H)$ then 
\[ \frac{1}{4}\|T^{*}T+TT^{*}\|\leq w^2(T)\leq \frac{1}{2}\|T^{*}T+TT^{*}\|\]  and
\[ w(T)\leq \frac{1}{2}(\|T\|+\|T^2\|^{\frac{1}{2}}).\]
Also Abu-Omar and Kittaneh in \cite{AOFK} proved that if $T\in B(H)$ then 
\begin{eqnarray*}
 \frac{1}{2}\sqrt{\|T^{*}T+TT^{*}\|+2m(T^2)} &\leq& w(T) \\
&\leq &  \frac{1}{2}\sqrt{\|T^{*}T+TT^{*}\|+2w(T^2)}.
\end{eqnarray*}
They also proved that the upper bound of $w(T)$ in this inequality is better than the above upper bounds in \cite{FK3, FK2}. 

Here we present an upper bound for the numerical radius of  $2\times 2$  operator matrices which improves the existing bound in \cite{AF}. Also we present  lower bounds for the numerical radius of  $2\times 2$ operator matrices. As an application of the results obtained we estimate  bounds for the zeros of a complex polynomial. Also, we show with numerical examples that the bounds obtained by us for the zeros of a polynomial  improves on the existing bounds.

\section{Main results}
We begin this section with the following lemmas which are used to reach our goal in this present article. These four lemmas can be found in \cite{AF, A, FH}.
\begin{lemma}[\cite{AF}] \label{theorem:lemma1}
	Let $T \in B(H)$, then $w(T)=\sup_{\theta \in \mathbb{R}}\| \textit {Re}(e^{i \theta}T) \|$.
\end {lemma}		
		
\begin {lemma}[\cite{AF}] \label{theorem:lemma2} 
Let $X \in B(H_1), Y \in B(H_2, H_1), Z \in B(H_1, H_2)$ and $W \in B(H_2)$. Then the following results hold:\\
$(i)$ $w\left(\begin{array}{cc}
    X&0 \\
    0&W
	\end{array}\right)=\max\{w(X), w(W)\}$.\\
$(ii)$ $w\left(\begin{array}{cc}
    0&Y \\
    Z&0
	\end{array}\right)=w\left(\begin{array}{cc}
    0&Z \\
    Y&0
	\end{array}\right)$.\\
$(iii)$ $w\left(\begin{array}{cc}
    0&Y \\
    Z&0
	\end{array}\right)= \sup_{\theta \in \mathbb{R}} \frac{1}{2}\|e^{i\theta}Y+e^{-i\theta}Z^{*}\|$.\\
$(iv)$ If  ~~$H_1=H_2$, then~~ $w\left(\begin{array}{cc}
    0&Y \\
    Y&0
	\end{array}\right)=w(Y)$.
\end {lemma}
\begin {lemma} [\cite{FH}]\label{theorem:lemma3}
Let $C,T\in B(H)$. Then $w(TC+C^{*}T)\leq 2w(T)$, where $C$ is any contraction $(\textit{i.e.}, \|C\|\leq 1)$.
\end {lemma}
\begin {lemma} [\cite{A}]\label{lemma:lemma4}

If $D=\left(\begin{array}{cccccc}
    0&0&.&.&.&0\\
		1&0&.&.&.&0\\
    .& & & & & \\
    .& & & & &\\
    .& & & & &\\
    0&0&.&.&1&0
   \end{array}\right)_{n,n}$ then  $w(D)=\cos\frac{\pi}{n+1}$.
\end{lemma}

Now we are ready to prove the following inequality for the numerical radius of  $2\times 2$ operator matrices which improves on the existing inequalities.

\begin{theorem} \label{theorem:1}
Let $X \in B(H_2, H_1), Y \in B(H_1, H_2)$.  Then 

\[w^4\left(\begin{array}{cc}
    0&X \\
    Y&0
	\end{array}\right)\leq \frac{1}{16}\|S\|^2+\frac{1}{4}w^2(YX)+\frac{1}{8}w(YXS+SYX),\]   where $S=|X|^2+|Y^{*}|^2$. 
\end{theorem}
\begin{proof}
Let $f(\theta)=\frac{1}{2}\|e^{i\theta}X+e^{-i\theta}Y^{*}\|$. Therefore, 
\begin{eqnarray*}
f(\theta)&=&\frac{1}{2}\|(e^{i\theta}X+e^{-i\theta}Y^{*})^{*}(e^{i\theta}X+e^{-i\theta}Y^{*})\|^{\frac{1}{2}}\\
         &=&\frac{1}{2}\|(e^{-i\theta}X^{*}+e^{i\theta}Y)(e^{i\theta}X+e^{-i\theta}Y^{*})\|^{\frac{1}{2}}\\
				 &=& \frac{1}{2}\|S+ 2\textit{Re}(e^{2i\theta}YX)\|^{\frac{1}{2}}\\
				 &=& \frac{1}{2}\|(S+ 2\textit{Re}(e^{2i\theta}YX))^2\|^{\frac{1}{4}}\\
				 &=& \frac{1}{2}\|S^2+ 4(\textit{Re}(e^{2i\theta}YX))^2+2\textit{Re}(e^{2i\theta}(YXS+SYX))\|^{\frac{1}{4}}\\
	\Rightarrow f^4(\theta)&\leq& \frac{1}{16}\|S\|^2+\frac{1}{4}\|\textit{Re}(e^{2i\theta}YX)\|^2+\frac{1}{8}	\|\textit{Re}(e^{2i\theta}(YXS+SYX))\|.		
\end{eqnarray*}
Now taking supremum over $\theta \in \mathbb{R}$ in the above inequality and then from Lemma \ref{theorem:lemma2} $(iii)$ and Lemma \ref {theorem:lemma1} we get,
\[w^4\left(\begin{array}{cc}
    0&X \\
    Y&0
	\end{array}\right)\leq \frac{1}{16}\|S\|^2+\frac{1}{4}w^2(YX)+\frac{1}{8}w(YXS+SYX).\] This completes the proof.
\end{proof}

 Now using Lemma \ref{theorem:lemma2} $(ii)$ and Theorem \ref{theorem:1} we get the following inequality.

\begin{cor}\label{theorem:cor1} 
Let $X \in B(H_2, H_1), Y \in B(H_1, H_2)$.  Then 
\[w^4\left(\begin{array}{cc}
    0&X \\
    Y&0
	\end{array}\right)\leq \frac{1}{16}\|P\|^2+\frac{1}{4}w^2(XY)+\frac{1}{8}w(XYP+PXY),\] where $P=|X^{*}|^2+|Y|^2$.
\end{cor}
Again using  Lemma \ref{theorem:lemma2} $(i)$ and Theorem \ref{theorem:1} we get the following inequality.

\begin{cor}\label{theorem:cor5}
Let $X \in B(H_2, H_1), Y \in B(H_1, H_2)$.  Then 

\[ w(XY)\leq \frac{1}{4} \sqrt{\|S\|^2+4w^2(YX)+2w(YXS+SYX)}\] where $S=|X|^2+|Y^{*}|^2$.
 \end{cor}
\begin{proof} Now,
\begin{eqnarray*}
w(XY) &\leq& \max \{w(XY), w(YX)\} \\
      &=& w\left(\begin{array}{cc}
    XY&0 \\
    0&YX
	\end{array}\right)\\
	&=& w\left(\begin{array}{cc}
    0&X \\
    Y&0
	\end{array}\right)^2\\
	&\leq& w^2\left(\begin{array}{cc}
    0&X \\
    Y&0
	\end{array}\right)\\
	&\leq& \frac{1}{4} \sqrt{\|S\|^2+4w^2(YX)+2w(YXS+SYX)}.
\end{eqnarray*}
\end{proof}

\begin{remark}
Using  Lemma \ref{theorem:lemma3}, it is easy to observe that the bound obtained  in Theorem \ref{theorem:1} is better than the second inequality  in \cite[Th. 3]{AF}.   
\end{remark}

\begin{remark} Here we note that when $H_1=H_2$ and $Y=X$ then it follows from Theorem \ref{theorem:1} and Lemma \ref{theorem:lemma2} $(iv)$ that $w^4(X) \leq \frac{1}{16}\|R\|^2+\frac{1}{4}w^2(X^2)+\frac{1}{8}w(X^2R+RX^2),$ where $R=|X|^2+|X^{*}|^2$. This inequality can be found in \cite[Th. 2.1]{BBP} 
\end{remark}

Next we prove a lower bound for the numerical radius of $2\times 2$ operator matrices.
\begin{theorem} \label{theorem:2}
Let $X \in B(H_2, H_1), Y \in B(H_1, H_2)$.  Then 

\[w^4\left(\begin{array}{cc}
    0&X \\
    Y&0
	\end{array}\right)\geq \frac{1}{16}\|S\|^2+\frac{1}{4}c^2(YX)+\frac{1}{8}m(YXS+SYX),\]   where $S=|X|^2+|Y^{*}|^2, ~c(YX)=\inf_{\theta \in \mathbb{R}}\inf_{ x\in H_2,\|x\|=1}\|\textit{Re}(e^{i\theta}YX)x\|$.  
\end{theorem}
\begin{proof}
Let $x\in H_2$ with $\|x\|=1$ and  $\theta$ be a real number such that $e^{2i\theta}\langle(YXS+SYX)x,x\rangle=|\langle(YXS+SYX)x,x\rangle|$. Then from  Lemma \ref{theorem:lemma2} $(iii)$  we get,
\begin{eqnarray*}
w\left(\begin{array}{cc}
    0&X \\
    Y&0
\end{array}\right)&\geq& \frac{1}{2}\|e^{i\theta}X+e^{-i\theta}Y^{*}\| \\
                  &\geq&\frac{1}{2}\|(e^{i\theta}X+e^{-i\theta}Y^{*})^{*}(e^{i\theta}X+e^{-i\theta}Y^{*})\|^{\frac{1}{2}}\\
                  &\geq&\frac{1}{2}\|(e^{-i\theta}X^{*}+e^{i\theta}Y)(e^{i\theta}X+e^{-i\theta}Y^{*})\|^{\frac{1}{2}}\\
				          &\geq& \frac{1}{2}\|S+ 2\textit{Re}(e^{2i\theta}YX)\|^{\frac{1}{2}}\\
				          &\geq& \frac{1}{2}\|(S+ 2\textit{Re}(e^{2i\theta}YX))^2\|^{\frac{1}{4}}\\
				          &\geq& \frac{1}{2}\|S^2+ 4(\textit{Re}(e^{2i\theta}YX))^2+2\textit{Re}(e^{2i\theta}(YXS+SYX))\|^{\frac{1}{4}}\\
									&\geq& \frac{1}{2}|\langle\big(S^2+ 4(\textit{Re}(e^{2i\theta}YX))^2+2\textit{Re}(e^{2i\theta}(YXS+SYX))\big)x,x\rangle|^{\frac{1}{4}}\\
									&\geq& \frac{1}{2}|\langle S^2x,x\rangle + 4\langle(\textit{Re}(e^{2i\theta}YX))^2x,x\rangle+2\textit{Re}\langle(e^{2i\theta}(YXS+SYX))x,x\rangle|^{\frac{1}{4}}\\
									&=& \frac{1}{2}\big[\| Sx\|^2 + 4\|\textit{Re}(e^{2i\theta}YX)x\|^2+2|\langle(YXS+SYX)x,x\rangle\big]^{\frac{1}{4}}\\
									&\geq& \frac{1}{2}\big[\| Sx\|^2 + 4c^2(YX)+2m(YXS+SYX)\big]^{\frac{1}{4}}.
		\end{eqnarray*}
Now taking supremum over $x\in H_2$ with $\|x\|=1$ in the above inequality we get,
\[w^4\left(\begin{array}{cc}
    0&X \\
    Y&0
	\end{array}\right)\geq \frac{1}{16}\|S\|^2+\frac{1}{4}c^2(YX)+\frac{1}{8}m(YXS+SYX).\] This completes the proof.
\end{proof} 
Now using Lemma \ref{theorem:lemma2}$(ii)$ and Theorem \ref{theorem:2} we get the following inequality.

\begin{cor}\label{theorem:cor3} 
Let $X \in B(H_2, H_1), Y \in B(H_1, H_2)$.  Then 
\[w^4\left(\begin{array}{cc}
    0&X \\
    Y&0
	\end{array}\right)\geq \frac{1}{16}\|P\|^2+\frac{1}{4}c^2(XY)+\frac{1}{8}m(XYP+PXY),\] where $P=|X^{*}|^2+|Y|^2$.
\end{cor}

\begin{remark} Here we note that when $H_1=H_2$ and $Y=X$ then it follows from Theorem \ref{theorem:2} and Lemma \ref{theorem:lemma2} $(iv)$ that $w^4(X) \geq \frac{1}{16}\|R\|^2+\frac{1}{4}c^2(X^2)+\frac{1}{8}m(X^2R+RX^2),$ where $R=|X|^2+|X^{*}|^2$. Also, this inequality can be found in \cite[Th. 3.1]{BBP} 
\end{remark}

Next we state the following lemma which can be found in \cite{OFK}.
\begin {lemma}\label{lemma:lemma5}
Let $X, Y, Z, W \in B(H)$. Then

\[w\left(\begin{array}{cc}
    X&Y \\
    Z&W
	\end{array}\right)\geq w\left(\begin{array}{cc}
    0&Y \\
    Z&0
	\end{array}\right)\] and
	\[w\left(\begin{array}{cc}
    X&Y \\
    Z&W
	\end{array}\right)\geq w\left(\begin{array}{cc}
    X&0 \\
    0&W
	\end{array}\right).\]
\end{lemma}

Now we are ready to prove an upper bound and a lower bound for the numerical radius of an operator matrix 
$\left(\begin{array}{cc}
    X&Y \\
    Z&W
	\end{array}\right)$, where $X, Y, Z, W \in B(H)$.

\begin{theorem}\label{theorem:3}
Let $X, Y, Z, W \in B(H)$. Then
\[w\left(\begin{array}{cc}
    X&Y \\
    Z&W
	\end{array}\right)\leq \max\{ w(X), w(W)\} + [\frac{1}{16}\|S\|^2+\frac{1}{4}w^2(ZY)+\frac{1}{8}w(ZYS+SZY)]^{\frac{1}{4}}\]
	and 
\[w\left(\begin{array}{cc}
    X&Y \\
    Z&W
	\end{array}\right)\geq \max\{ w(X), w(W), [\frac{1}{16}\|S\|^2+\frac{1}{4}c^2(ZY)+\frac{1}{8}m(ZYS+SZY)]^{\frac{1}{4}}\},\]
where $S=|Y|^2+|Z^{*}|^2, c(ZY)=\inf_{\theta \in \mathbb{R}}\inf_{ x\in H,\|x\|=1}\|\textit{Re}(e^{i\theta}ZY)x\|$.  
\end{theorem}
\begin{proof}
The proof  follows easily from the Theorem \ref{theorem:1}, Theorem \ref{theorem:2} and Lemma \ref{lemma:lemma5}.
\end{proof}

\section{Application}
Let us consider a monic polynomial $p(z)=z^n+a_{n-1}z^{n-1}+\ldots+a_1z+a_0$ with complex coefficients $a_0, a_1, \ldots, a_{n-1}$. Then the Frobenius companion matrix of $p(z)$ is given by  
\begin{eqnarray*}
  C(p)&=&\left(\begin{array}{ccccccc}
    -a_{n-1}&-a_{n-2}&.&.&.&-a_1&-a_0 \\
    1&0&.&.&.&0&0\\
		0&1&.&.&.&0&0\\
    .& & & & & & \\
    .& & & & & &\\
    .& & & & & &\\
    0&0&.&.&.&1&0
    \end{array}\right).
		\end{eqnarray*}
Then the zeros of the polynomial $p(z)$ are  exactly the eigenvalues of $C(p)$. Also we know that the spectrum $\sigma(C(p))\subseteq \overline {W(C(p))}$, so that if $z$ is a zero of the polynomial  $p(z)$ then $|z|\leq w(C(p))$. 

Many mathematicians have estimated  the zeros of the polynomial using this above argument, some of them are mentioned below. Let $\lambda$ be a zero of the polynomial  $p(z)$.\\
  (1) Cauchy \cite{CM}	proved that 
		   \[ |\lambda|\leq 1+\max \{ |a_0|, |a_1|, \ldots, |a_{n-1}|\}.\]
	(2) Carmichael and Mason \cite{CM} proved that 
		\[ |\lambda|\leq (1+|a_0|^2+|a_1|^2+\ldots+|a_{n-1}|^2)^{\frac{1}{2}}.\]
	(3) Montel \cite{CM}	proved that 
		   \[ |\lambda|\leq \max \{ 1, |a_0|+ |a_1|+ \ldots+ |a_{n-1}|\}.\]
	(4) Fujii and Kubo \cite{FK1} proved that 
			\[|\lambda|\leq \cos\frac{\pi}{n+1}+\frac{1}{2}\big[\big(\sum_{j=0}^{n-1}|a_j|^2\big)^{\frac{1}{2}}+|a_{n-1}|\big].\]
	(5) Alpin et. al. \cite{YML} proved that
	     \[ |\lambda| \leq  \max_{1\leq k \leq n}[(1+|a_{n-1}|)(1+|a_{n-2}|)\ldots(1+|a_{n-k}|)]^{\frac{1}{k}}.\]
	(6) Paul and Bag \cite{PB2} proved that 
	\[ |\lambda| \leq\frac{1}{2}\big[w(A)+\cos\frac{\pi}{n-1}+\sqrt{(w(A)-\cos\frac{\pi}{n-1})^2+(1+\sqrt{\sum_{k=3}^n|a_{n-k}|^2})^2}\big],\]where
$A=\left(\begin{array}{cc}
    -a_{n-1}&-a_{n-2} \\
    1&0
 \end{array}\right).$\\
	(7)   Abu-Omar and Kittaneh \cite{A} proved that 
	     \[|\lambda| \leq \frac{1}{2}\big[\frac{1}{2}(|a_{n-1}|+\alpha)+\cos\frac{\pi}{n+1}+\sqrt{(\frac{1}{2}(|a_{n-1}|+\alpha)-\cos\frac{\pi}{n+1})^2+4\alpha' }\big],\] where $\alpha=\sqrt{\sum_{j=0}^{n-1}|a_j|^2}$ and $\alpha'=\sqrt{\sum_{j=0}^{n-2}|a_j|^2}$.\\
	(8) M. Al-Dolat et. al. \cite{MKMF} proved that
       \[ \lambda \leq \max \{ w(A), \cos\frac{\pi}{n+1}\} +\frac{1}{2} \big(1+\sqrt{{\sum_{j=0}^{n-3}|a_j|^2}} \big), \] where
$A=\left(\begin{array}{cc}
    -a_{n-1}&-a_{n-2} \\
    1&0
 \end{array}\right).$\\
We first prove the following theorem.

\begin{theorem}\label{theorem:4}
Let $\lambda$ be any zero of $p(z)$. Then

\begin{eqnarray*} 
 |\lambda|&\leq& |\frac{a_{n-1}}{n}|+\cos\frac{\pi}{n}+\frac{1}{2} [(1+\alpha)^2+4\alpha+4\sqrt{\alpha}(1+\alpha)]^{\frac{1}{4}},  
\end{eqnarray*}
where 
\begin{eqnarray*}
\alpha_{r}&=&\sum^{n}_{k=r} {^k}C_r\big(-\frac{a_{n-1}}{n}\big)^{k-r} a_{k},  ~~ r=0,1,\ldots, n-2, a_n=1, {^0}C_0=1, \\
\alpha&=&\sum^{n-2}_{i=0}|\alpha_i|^2.
\end{eqnarray*}
\end{theorem}
\begin{proof}
Putting $z=\eta -\frac{a_{n-1}}{n}$ in the polynomial $p(z)$ we get,  a polynomial \\$q(\eta)=\eta^n+\alpha_{n-2}\eta^{n-2}+\alpha_{n-3}\eta^{n-3}+\ldots +\alpha_{1}\eta+\alpha_{0},$  \\ where $\alpha_{r}=\sum^{n}_{k=r} {^k}C_r\big(-\frac{a_{n-1}}{n}\big)^{k-r} a_{k},  ~~ r=0,1,\ldots, n-2$, $a_n=1$ and $^{0}C_0=1$.

Now the  Frobenius companion matrix of the polynomial $q(\eta)$ is $C(q)=\left(\begin{array}{cc}
    A&B \\
    C&D
	\end{array}\right)$ where  $A=(0)_{1,1}$,  $B=(-\alpha_{n-2} ~~ -\alpha_{n-3}~~ \ldots~~ -\alpha_{1} ~~ -\alpha_{0})_{1,n-1}$,  $C^t=(1 ~~0 ~~\ldots ~~0~~ 0)_{1,n-1} $,  $D=\left(\begin{array}{cccccc}
    0&0&.&.&.&0\\
		1&0&.&.&.&0\\
    .& & & & & \\
    .& & & & &\\
    .& & & & &\\
    0&0&.&.&1&0
   \end{array}\right)_{n-1,n-1}.$
	
	Now using Lemma \ref{theorem:lemma2}$(i)$ and Lemma \ref{lemma:lemma4} we get,
	\begin{eqnarray*} 
	w\left(\begin{array}{cc}
    A&B \\
    C&D
	\end{array}\right)&\leq&  w\left(\begin{array}{cc}
    A&0 \\
    0&D
	\end{array}\right)+w\left(\begin{array}{cc}
    0&B \\
    C&0
	\end{array}\right)\\
	&=& w(D)+w\left(\begin{array}{cc}
    0&B \\
    C&0
	\end{array}\right)\\
	&=& \cos\frac{\pi}{n}+w\left(\begin{array}{cc}
    0&B \\
    C&0
	\end{array}\right).
	\end{eqnarray*}
	Therefore, if $\eta$ is any zero of the polynomial $q(\eta)$ then $|\eta|\leq \cos\frac{\pi}{n}+w\left(\begin{array}{cc}
    0&B \\
    C&0
	\end{array}\right).$ 
	Therefore if $\lambda$ is any zero of the polynomial $p(z)$ then $|\lambda|\leq |\frac{a_{n-1}}{n}|+\cos\frac{\pi}{n}+w\left(\begin{array}{cc}
    0&B \\
    C&0
	\end{array}\right).$
	Now using the Theorem \ref{theorem:1} in the above inequality we get, 
	\begin{eqnarray*}
	 |\lambda|&\leq& |\frac{a_{n-1}}{n}|+\cos\frac{\pi}{n}+[\frac{1}{16}\|S\|^2+\frac{1}{4}w^2(CB)+\frac{1}{8}w(CBS+SCB)]^{\frac{1}{4}}, \\ &&\mbox {where}~~ S=B^{*}B+CC^{*} \\ 
	          &\leq& |\frac{a_{n-1}}{n}|+\cos\frac{\pi}{n}+\frac{1}{2}[\|S\|^2+ 4\|B\|^2+ 4 \|B\|\|S\|]^{\frac{1}{4}}, ~[~\mbox {using Lemma }\ref{theorem:lemma3}]\\ 
						&\leq& |\frac{a_{n-1}}{n}|+\cos\frac{\pi}{n}+\frac{1}{2} [(1+\alpha)^2+4\alpha+4\sqrt{\alpha}(1+\alpha)]^{\frac{1}{4}}
	\end{eqnarray*}
This completes the proof of the theorem.
\end{proof}

\begin{remark}
There is an error in the summation formula for  computation of the coefficient $b_r$ in \cite[page 237]{PB2} and $\alpha_r$ in \cite[Th. 3.1]{PB1}.
\end{remark}

  We illustrate with numerical examples to show that the above bound obtained by us in Theorem \ref{theorem:4} is better than the existing bounds. 
	\begin{example}
	Consider the polynomial $p(z)=z^5+2z^4+z+1$. Then the upper bounds of  the zeros of this polynomial $p(z)$ estimated by different mathematicians are as shown in the following table.
	\end{example}
	\begin{center}
\begin{tabular}{ |c|c| } 
 \hline
    Cauchy \cite{CM} &  3.000 \\ 
 \hline
    Montel \cite{CM} &  4.000 \\ 
 \hline
  Carmichael and Mason \cite{CM}& 2.645 \\
 \hline 
   Fujii and Kubo \cite{FK1} & 3.090  \\
 \hline
Alpin et. al. \cite{YML} & 3.000 \\
\hline
Paul and Bag \cite{PB2} & 2.810 \\
\hline
Abu-Omar and Kittaneh \cite{A} & 2.914 \\
\hline
    M. Al-Dolat et. al. \cite{MKMF} &  3.325 \\ 
\hline
 \end{tabular}
\end{center}
But our bound obtained in Theorem \ref{theorem:4} gives $|\lambda| \leq 2.625$ which is better than all the estimations mentioned above. 

\bibliographystyle{amsplain}

\end{document}